\documentclass[12pt]{amsart}
\pdfoutput=1
\usepackage[headings]{fullpage}
\usepackage{graphicx}
\usepackage{url}
\usepackage[vlined,boxruled,norelsize]{algorithm2e}

\usepackage[bookmarks=true,%
    colorlinks=true,%
    linkcolor=blue,%
    citecolor=blue,%
    filecolor=blue,%
    menucolor=blue,%
    urlcolor=blue,%
    breaklinks=true]{hyperref}

\newtheorem{theorem}{Theorem}[section]
\newtheorem{lemma}[theorem]{Lemma}
\theoremstyle{definition}
\newtheorem{definition}[theorem]{Definition}
\newtheorem{remark}[theorem]{Remark}
\newtheorem{example}[theorem]{Example}

\catcode`\@=11
\long\def\@makecaption#1#2{%
     \vskip 10pt

\setbox\@tempboxa\hbox{
       \small\sf{\bfcaptionfont #1. }\ignorespaces #2}%
     \ifdim \wd\@tempboxa >\captionwidth {%
         \rightskip=\@captionmargin\leftskip=\@captionmargin
         \unhbox\@tempboxa\par}%
       \else
         \hbox to\hsize{\hfil\box\@tempboxa\hfil}%
     \fi}
\font\bfcaptionfont=cmssbx10 scaled \magstephalf
\newdimen\@captionmargin\@captionmargin=2\parindent
\newdimen\captionwidth\captionwidth=\hsize
\catcode`\@=12

\SetAlgorithmName{Pseudocode}{List of Pseudocode}

\newcommand\algorithmref[1]{Pseudocode~\ref{#1}}

\newcommand\BZ{\mathbb{Z}}
\newcommand\BH{\mathbb{H}}
\newcommand\BS{\mathbb{S}}
\newcommand\BE{\mathbb{E}}

\newcommand\notCensus{$\mbox{}^*$}

\newcommand\Comm{\operatorname{Comm}}
\newcommand\Isom{\operatorname{Isom}}

\newcommand\mfdMarker[1]{\mathrm{\mathcal{#1}}}
\newcommand\MunivPrinCong[2]{\mfdMarker{U}^{#1}_{#2}}

\newcommand\forceLongestBibliographyLabelToFixBibTeXBug[1]{
    \let\stdthebibliography\thebibliography
    \let\stdendthebibliography\endthebibliography
    \renewenvironment*{thebibliography}[1]%
            {\stdthebibliography{#1}}
            {\stdendthebibliography}
}

\begin{document}

\title[A census of hyperbolic Platonic manifolds and aug. knotted trivalent graphs]{A census of hyperbolic Platonic manifolds\\ and augmented knotted trivalent graphs\\~\\Appendix:\\
Hyperbolic ideal cubulations can be subdivided\\ into ideal geometric triangulations}
\author[Goerner]{Matthias Goerner} 
\address{Pixar Animation Studios\\ 
         1200 Park Avenue\\ 
         Emeryville, CA 94608, USA \newline
         {\tt \url{http://www.unhyperbolic.org/}}}
\email{enischte@gmail.com}

\date{March 8th, 2017}

\begin{abstract}
We call a 3-manifold {\em Platonic} if it can be decomposed into isometric Platonic solids. Generalizing an earlier publication by the author and others where this was done in case of the hyperbolic ideal tetrahedron, we give a census of hyperbolic Platonic manifolds and all of their Platonic tessellations. For the octahedral case, we also identify which manifolds are complements of an augmented knotted trivalent graph and give the corresponding link. A (small version of) the Platonic census and the related improved algorithms have been incorporated into \texttt{SnapPy}. The census also comes in \texttt{Regina} format.
\end{abstract}

\maketitle

\vspace{-0.5cm}

\setcounter{tocdepth}{1}

\tableofcontents

\vspace{-1.1cm}

\section{Introduction}

\subsection{Platonic manifolds}

We call a spherical, Euclidean, or hyperbolic 3-manifold {\em Platonic} if it can be decomposed into isometric finite or ideal Platonic solids. We call such a decomposition a {\em Platonic tessellation}. There exists Platonic manifolds that admit more than one Platonic decomposition, thus we use the two terms Platonic manifold and Platonic tessellation to distinguish whether we regard objects as isomorphic when they are isometric as manifolds or combinatorially isomorphic as tessellations. The goal of this paper is to create a census of such manifolds and tessellations.

It is motivated by the fact that many manifolds that have played a key role in the development of low-dimensional topology are Platonic. Examples include the Seifert-Weber space (which has a ``homology sister'' obtained by gluing the same hyperbolic dodecahedron, see Section~\ref{sec:seifertWeber}) as well as exactly three knot complements (the figure-eight knot and the two dodecahedral knots in \cite{AR:dodecahedral}, see \cite{reid:arithmeticity,hoffman:dodecahedral}) and many link complements such as the complement of the Whitehead link and the Borromean rings. Furthermore, Baker showed that each link is a sublink with octahedral and, thus, Platonic complement \cite{Baker:AllLinksAreSublinksOfArithmeticLinks}. This also follows from van der Veen's work \cite{Roland:AKTG} showing that the complement of an augmented knotted trivalent graph ({\em AugKTG}\/) is octahedral and in Section~\ref{sec:AugKTG} we enumerate AugKTGs up to complements with 8 octahedra.

Two examples of Platonic manifolds that exhibit many symmetries are the complements of the minimally twisted 5-component chain link and the Thurston congruence link \cite{Th:see,Agol:CongruenceLink}. Both are principal congruence manifolds as well as regular tessellations in the sense of Definition~\ref{def:RegTess}. Baker and Reid enumerated all known principal congruence links \cite{bakerReid:prinCong} and
the author showed that there are at most 21 link complements admitting a regular tessellation \cite{Goerner:RegTessLinkComps}.

The census of Platonic manifolds and tessellations illustrates a number of interesting phenomena such as commensurability (in particular, of tetrahedral and cubical manifolds) and the difference between arithmetic and non-arithmetic manifolds with implications on hidden symmetries and the existence of Platonic manifolds admitting more than one Platonic tessellation, which we will discuss in Section~\ref{sec:propPlat}.

The author and others previously provided such a census in the case of the tetrahedron \cite{FGGTV:tetrahedralCensus}. Everitt did similar work but considered manifolds consisting of only a single Platonic solid \cite{Everitt:manifoldsFromPlatonicSolids}.

\subsection{Results}

A {\em tetrahedral}, {\em octahedral}, {\em icosahedral}, {\em cubical}, or {\em dodecahedral} tessellation or manifold is a hyperbolic Platonic tessellation or manifold made from the respective Platonic solid. We call it {\em closed} or {\em cusped} depending on whether the vertices of the solid are finite or ideal. Unless prefixed by {\em right-angled}, the term {\em closed dodecahedral} tessellation or manifold exclusively refers to the case where 5 (not necessarily distinct) dodecahedra are adjacent to an edge and geometrically have a dihedral angle of $2\pi/5$, i.e., $\{5,3,5\}$ in the notation introduced in Section~\ref{subsec:RegTess}. We will not cover right-angled closed dodecahedral tessellations $\{5,3,4\}$ since they are dual to closed cubical tessellations $\{4,3,5\}$, see Table~\ref{tbl:modelRegTess}.

\begin{theorem}\label{thm:Main}
The numbers of hyperbolic Platonic tessellations and manifolds up to a certain number of Platonic solids are listed in Table~\ref{table:census336}, \ref{table:census344}, \ref{table:census436}, \ref{table:census536}, \ref{table:census353}, \ref{table:census435}, and \ref{table:census535}.  Table~\ref{table:census344} also lists the number of octahedral manifolds that are complements of AugKTGs.
\end{theorem}

Since the total number of manifolds exceeds one million, we could not include all of them in SnapPy \cite{SnapPy}. We thus distinguish between the {\em small} and the {\em large} census of hyperbolic Platonic tessellations, respectively, manifolds with only the latter one including those tessellations and manifolds marked with a star in Table~\ref{table:census344} and~\ref{table:census436}. The small census is part of a SnapPy installation, beginning with version 2.4. The small and large census are also both available at \cite{goerner:platonicCensusData}. Section~\ref{sec:results} gives details about the naming conventions and examples of how to access each census. Section~\ref{sec:AugKTG} shows how to access the link diagrams for AugKTGs. Section~\ref{section:tessTools} illustrates a tool to query a Platonic manifold about various properties.

\begin{table}
\caption{Cusped Tetrahedral Census $\{3,3,6\}$. Included from \cite{FGGTV:tetrahedralCensus} for completeness.}
\label{table:census336}
\bigskip
\begin{tabular}{r||r|r||r|r||r}
\multicolumn{1}{l||}{$\{3,3,6\}$}
& \multicolumn{2}{|c||}{cusped tetrahedral} & \multicolumn{2}{|c||}{cusped tetrahedral} 
& \multicolumn{1}{c}{homology} \\
& \multicolumn{2}{|c||}{tessellations} & \multicolumn{2}{|c||}{manifolds} 
& \multicolumn{1}{c}{links} \\
Tetrahedra & orientable & non-or. & orientable & non-or. & \\ \hline \hline
1 & 0 & 1 & 0 & 1 & 0 \\ \hline
2 & 2 & 2 & 2 & 1 & 1 \\ \hline
3 & 0 & 1 & 0 & 1 & 0 \\ \hline
4 & 4 & 4 & 4 & 2 & 2 \\ \hline
5 & 2 & 12 & 2 & 8 & 0 \\ \hline
6 & 7 & 14 & 7 & 10 & 0 \\ \hline
7 & 1 & 1 & 1 & 1 & 0 \\ \hline
8 & 14 & 10 & 13 & 6 & 5 \\ \hline
9 & 1 & 6 & 1 & 6 & 0 \\ \hline
10 & 57 & 286 & 47 & 197 & 12 \\ \hline
11 & 0 & 17 & 0 & 17 & 0 \\ \hline
12 & 50 & 117 & 47 & 80 & 7 \\ \hline
13 & 3 &8 & 3 & 8 & 0 \\ \hline
14 & 58 &134 & 58 & 113 & 25 \\ \hline
15 & 91 & 975 & 81 & 822 & 0 \\ \hline
16 & 102 & 175 & 96 & 142 & 32 \\ \hline
17 & 8 & 52 & 8 & 52 & 0 \\ \hline
18 & 213 & 1118 & 199 & 810 & 66 \\ \hline
19 & 25 & 326 & 25 & 326 & 0 \\ \hline
20 & 1886 & 26320 & 1684 & 22340 & 209 \\ \hline
21 & 31 & 251 & 31 & 251 & 0 \\ \hline
22 & 390 &- & 381 & - & 148 \\ \hline
23 & 58 &- & 58 & - & 0 \\ \hline
24 & 1544 & - & 1465 & - & 378 \\ \hline
25 & 7563 & - & 7367 & - & 0
\end{tabular}
\end{table}

\begin{table}
\caption{Cusped Octahedral Census $\{3,4,4\}$.}
\label{table:census344}
\bigskip
\begin{center}
\begin{tabular}{r||r|r||r|r||r|r}
\multicolumn{1}{l||}{$\{3,4,4\}$}& \multicolumn{2}{|c||}{cusped octahedral} & \multicolumn{2}{|c||}{cusped octahedral} 
& \multicolumn{1}{c|}{homology} & \multicolumn{1}{c}{Aug} \\
& \multicolumn{2}{|c||}{tessellations} & \multicolumn{2}{|c||}{manifolds} 
& \multicolumn{1}{c|}{links} & \multicolumn{1}{c}{KTG} \\
Octahedra & orientable & non-or. & orientable & non-or. & & \multicolumn{1}{c}{} \\ \hline \hline
1 & 2 & 11 & 2 & 6 & 2\\
2 & 27 & 117 & 21 & 62 & 9 & 4\\
3 & 29 & 324 & 24 & 208 & 11\\
4 & 446 & 4585 & 351 & 3076 & 83 & 24\\
5 & 353 & 19372 & 294 & 16278 & 119\\
6 & 8339 & \notCensus 250692 & 7524 & \notCensus 218397 & 849 & 210\\
7 & 3549 &   & 3056 & & 1029\\
8 & \notCensus 452445 &   & \notCensus440773 & & 12186 &  2821\\
\end{tabular}
\end{center}
\end{table}

\begin{table}
\caption{Cusped Cubical Census $\{4,3,6\}$.}
\label{table:census436}
\bigskip
\begin{center}
\begin{tabular}{r||r|r||r|r||r}
\multicolumn{1}{l||}{$\{4,3,6\}$} & \multicolumn{2}{|c||}{cusped cubical} & \multicolumn{2}{|c||}{cusped cubical} 
& \multicolumn{1}{c}{homology} \\
& \multicolumn{2}{|c||}{tessellations} & \multicolumn{2}{|c||}{manifolds} 
& \multicolumn{1}{c}{links} \\
Cubes & orientable & non-or. & orientable & non-or. & \\ \hline \hline
1 & 3 & 8 & 3 & 7 & 0\\
2 & 45 & 163 & 45 & 145 & 5\\
3 & 64 & 559 & 61 & 519 & 0 \\
4 & 704 & 9274 & 685 & 8795 & 29\\
5 & 778 & 31630 & 747 & 30948 & 0\\
6 & 9517 & \notCensus 529485 & 9267 & \notCensus 519385 & 239\\
7 & 23298 & & 22887 & & 0 \\
\end{tabular}
\end{center}
\end{table}

\begin{table}
\caption{Cusped Dodecahedral Census $\{5,3,6\}$.}
\label{table:census536}
\bigskip
\begin{center}
\begin{tabular}{r||r|r||r}
\multicolumn{1}{l||}{$\{5,3,6\}$} & \multicolumn{2}{|c||}{cusped dodecahedral} & \multicolumn{1}{c}{homology} \\
& \multicolumn{2}{|c||}{tessellations/manifolds}  
& \multicolumn{1}{c}{links} \\
Dodecahedra & orientable & non-or. & \\ \hline \hline
1 & 10 & 67 & 0 \\
2 & 915 & 4079 & 156\\
\end{tabular}
\end{center}
\end{table}

\begin{table}
\caption{Closed Icosahedral Census $\{3,5,3\}$ (dual tessellations counted only once).}
\label{table:census353}
\bigskip
\begin{center}
\begin{tabular}{r||r|r||r|r}
\multicolumn{1}{l||}{$\{3,5,3\}$} & \multicolumn{2}{|c||}{closed icosahedral} & \multicolumn{2}{|c}{closed icosahedral} 
 \\
& \multicolumn{2}{|c||}{tessellations} & \multicolumn{2}{|c}{manifolds} 
 \\
Icosahedra & orientable & non-orientable & orientable & non-orientable \\ \hline \hline
1 & 6 & 0 & 6 & 0 \\
2 & 5 & 1 & 5 & 1  \\
3 & 3 & 0 & 3 & 0 \\
4 & 15 & & 15 
\end{tabular}
\end{center}
\end{table}

\begin{table}
\caption{Closed Cubical Census $\{4,3,5\}$.}
\label{table:census435}
\bigskip
\begin{center}
\begin{tabular}{r||r|r||r|r}
\multicolumn{1}{l||}{$\{4,3,5\}$} & \multicolumn{2}{|c||}{closed cubical} & \multicolumn{2}{|c}{closed cubical} 
 \\
& \multicolumn{2}{|c||}{tessellations} & \multicolumn{2}{|c}{manifolds} 
 \\
Cubes & orientable & non-orientable & orientable & non-orientable \\ \hline \hline
5 & 10 & 4 & 10 & 2  \\
10 & 68 & 150 & 59 & 91
\end{tabular}
\end{center}
\end{table}

\begin{table}
\caption{Closed Dodecahedral Census $\{5,3,5\}$ (dual tessellations counted only once).}
\label{table:census535}
\bigskip
\begin{center}
\begin{tabular}{r||r|r||r|r}
\multicolumn{1}{l||}{$\{5,3,5\}$} & \multicolumn{2}{|c||}{closed dodecahedral} & \multicolumn{2}{|c}{closed dodecahedral} 
 \\
& \multicolumn{2}{|c||}{tessellations} & \multicolumn{2}{|c}{manifolds} 
 \\
Dodecahedra & orientable & non-orientable & orientable & non-orientable\\ \hline \hline
1 & 9 & 0  & 8 & 0 \\
2 & 17 & 10 & 17 & 10 \\
3 & 52 & & 51  
\end{tabular}
\end{center}
\end{table}

\begin{table}
\caption{3-dimensional model regular tessellations.\label{tbl:modelRegTess}}
\bigskip
\begin{tabular}{c||c|c|c|c}
 & Spherical & Euclidean & Hyperbolic & Hyperbolic \\
 Solid &  & & closed & cusped \\ \hline \hline
Tetrahedron & $\{3,3,3\}$, $\{3,3,4\}$, $\{3,3,5\}$ & & & $\{3,3,6\}$ \\ \hline
Octahedron & $\{3,4,3\}$ & & & $\{3,4,4\}$ \\ \hline
Cube & $\{4, 3, 3\}$ & $\{4,3,4\}$ & $\{4,3,5\}$ & $\{4,3,6\}$ \\ \hline
Icosahedron & & & $\{3,5,3\}$ \\ \hline
Dodecahedron & $\{5,3,3\}$ & & $\{5,3,4\}$, $\{5,3,5\}$ & $\{5,3,6\}$
\end{tabular}
\end{table}

\clearpage

\subsection{Methods}

We first enumerate the Platonic tessellations and then group them by isometry type to enumerate the Platonic manifolds. 

For the enumeration of Platonic tessellations, we generalize the algorithm introduced in \cite{FGGTV:tetrahedralCensus} in Section~\ref{sec:algo}. The new algorithm uses the barycentric subdivision of a Platonic tessellation and a variant of the isomorphism signature \cite{burton:encode, burton:encode2} specialized to triangulations arising from such subdivisions to save memory since it is often the limiting factor when running the algorithm. The new algorithm is also multithreaded.

To group the Platonic tessellations by isometry type, we use different invariants for the cusped and the closed case, see Section~\ref{sec:enumPlatMan}. For the cusped case, we can simply group the tessellations by their isometry signature (see \cite[Definition~3.4]{FGGTV:tetrahedralCensus}) since it is a complete invariant of a cusped hyperbolic manifold. Since the publication of \cite{FGGTV:tetrahedralCensus}, the author has generalized the algorithm and incorporated various features into SnapPy that now make it easy to compute verified isometry signatures for any cusped hyperbolic manifold when running SnapPy inside Sage, see Section~\ref{sec:isometrySig}. This work included porting Burton's isomorphism signature to SnapPy, exposing the canonical retriangulation to Python, implementing verified computation of shape intervals inspired by HIKMOT \cite{hikmot}, improving SnapPy's recognition of number fields, and generalizing the code from \cite{DHL} to compute cusp cross sections and tilts given shapes as intervals or exact expressions.

Unfortunately, we do not have an equivalent of the isometry signature for closed hyperbolic manifolds. Fortunately, the number of closed Platonic tessellations is fairly small and we can try various invariants to group the tessellations. We can then verify that the chosen invariant was strong enough to separate all the non-isometric Platonic manifolds by finding isometries between all manifolds in a group. As invariant, we picked the list of first homologies of covering spaces of a given manifold up to a certain degree, see Section~\ref{sec:closedInvariant}.

For the enumeration of AugKTGs in Section~\ref{sec:AugKTG}, we recursively perform the moves generating AugKTGs. Many sequences of moves result in the same planar projection of the same AugKTG making enumeration prohibitively expensive unless we have a method to detect whether a similar planar projection has already been enumerated and stop recursing. For this we develop an isomorphism signature of fat graphs and planar projections based on ideas similar to Burton's isomorphism signature for triangulations \cite{burton:encode, burton:encode2}.

\subsection{Relation to regular tessellation} \label{subsec:RegTess}

For completeness, this section reviews some well-known concepts. We mostly follow existing literature, but give the term ``regular tessellation'' a more general meaning (dropping the assumption that the underlying space is simply connected), define ``local regular tessellation'', and distinguish between a tessellation ``hiding symmetries'' (see Definition~\ref{def:hiddenSymTess}) and an orbifold ``admitting hidden symmetries''.

Recall the regular tessellations of $\BS^n, \BE^n,$ and $\BH^n$ by finite-volume regular polytopes with the defining property that their isometry group acts transitively on flags, see Table~\ref{tbl:modelRegTess}. We call these {\em model regular tessellations} and generalize the notion of regular tessellation as follows:

\begin{definition} \label{def:RegTess}
A {\em regular tessellation} is a tessellation of manifold $M$ into finite or ideal polytopes such that any (right-handed if $M$ orientable) flag can be taken to any other such flag through a combinatorial isomorphism.
\end{definition}

\begin{example}
2-dimensional regular tessellations are also called ``regular maps'' in the literature and were classified by Conder up to genus 101 \cite{conder:regmaps,conder:regmaps2}. 3-dimensional regular tessellations include the Poincare homology sphere, the Seifert-Weber space as well as the regular tessellation link complements in \cite{Goerner:RegTessLinkComps}.
\end{example}

Recall that each regular tessellation of dimension $n$ has an invariant called the {\em Schl\"afli symbol} defined inductively. For $n=1$, let $\{p_1\}$ denote the regular $p_1$-gon. The polytopes of a regular tessellation of dimension $n$ are all the same and are all regular in the sense that the faces of a polytope form a regular tessellation of dimension $n-1$, thus we obtain a Schl\"aflisymbol $\{p_1,\dots, p_{n-1}\}$ for them. Similarly, every $(n-2)$-cell of a regular tessellation of dimension $n$ has the same order $p_n$, which is the number of (not necessarily distinct) polytopes adjacent to it. By adding this number to the Schl\"aflisymbol for the polytopes, we obtain the Schl\"aflisymbol $\{p_1, \dots, p_n\}$ for a regular tessellation of dimension $n$.

The Schl\"aflisymbol for the vertex link of such a regular tessellation is $\{p_2,\dots, p_{n}\}$ and the dual regular tessellation has the Schl\"aflisymbol $\{p_1, \dots, p_n\}^* =\{p_n, \dots, p_1\}$.

\begin{example}
The cube has Schl\"afli symbol $\{4,3\}$ and the self-dual tessellation of $\BE^3$ by cubes has $\{4,3,4\}$.
\end{example}

\begin{definition}
A {\em locally regular tessellation} is a tessellation of a manifold $M$ such that its universal cover is a regular tessellation.
\end{definition}

Note that every regular tessellation is a locally regular tessellation. Also note that we can assign a locally regular tessellation the Schl\"afli symbol of its universal cover. In fact, at least in dimension 3, the locally regular tessellations are exactly those tessellations with a well-defined Schl\"afli symbol in the following sense: a tessellation is locally regular if and only if all polytopes are the same and are regular and each edge has the same order.

\begin{example}
The tessellation of the figure-eight knot complement by two regular ideal tetrahedra is locally regular and has Schl\"afli symbol $\{3,3,6\}$ and is thus finite-volume as defined as follows.
\end{example}

\begin{definition}
We call a locally regular tessellation {\em finite-volume} if its universal cover is combinatorially isomorphic to a model regular tessellation by finite-volume polytopes, or equivalently, has the same Schl\"afli symbol as one of the model regular tessellations by finite-volume polytopes. We call a 3-dimensional finite-volume locally regular tessellation a {\em Platonic tessellation}.
\end{definition}

\begin{example}
Tessellations with Schl\"afli symbol $\{3,3,7\}$ or $\{6,3,3\}$ are not finite-volume.
\end{example}

A finite-volume locally regular tessellation determines a geometric structure on the underlying manifold $M$ that is unique (up to scaling in the Euclidean case) when requiring that all polytopes are regular and isometric. For a finite-volume locally regular tessellation, each combinatorial isomorphism induces an isometry of the induced geometric structure. The converse is in general false and thus we introduce the following notion.

\begin{definition} \label{def:hiddenSymTess}
We say that a finite-volume locally regular tessellation {\em hides symmetries} if there is an isometry of the induced geometric structure not coming from a combinatorial isomorphism.
\end{definition}

\begin{remark} \label{remark:admitHidSym}
Note that some literature (e.g., \cite{Walsh}) uses the term ``admitting hidden symmetry'' to refer to a different notion that is applied to an orbifold $O=\BH^3/\Gamma$ instead of a tessellation and that is defined in terms of the normalizer and commensurator of $\Gamma$. We shall see that these two notions are closely related in Section~\ref{sec:hiddenSym}.
\end{remark}

Let $\Gamma_{\{p_1,\dots,p_n\}}$, respectively, $\Gamma^+_{\{p_1,\dots,p_n\}}$ denote the symmetry, respectively, orientation-preserving symmetry group of the model regular tessellation $\{p_1,\dots,p_n\}$.
By definition, every hyperbolic finite-volume locally regular tessellation is the quotient of a model regular tessellation by a torsion-free subgroup $\Gamma\subset\Gamma_{\{p_1,\dots,p_n\}}$. This is why we chose to call them model regular tessellations in analogy to model geometries. A tessellation is regular if and only if $\Gamma\triangleleft \Gamma_{\{p_1,\dots,p_n\}}$ or $\Gamma\triangleleft \Gamma^+_{\{p_1,\dots,p_n\}}$. A tessellation hides a symmetry if $N(\Gamma)\setminus \Gamma_{\{p_1,\dots,p_n\}}$ is non-empty.

\section{The enumeration of Platonic tessellations} \label{sec:algo}

The algorithm to enumerate Platonic tessellations is based on the earlier algorithm to enumerate hyperbolic tetrahedral manifolds \cite{FGGTV:tetrahedralCensus}.

\subsection{Barycentric subdivision and specialized isomorphism signature} \label{sec:barySubdiv} \label{sec:specialIsoSig}

To generalize the algorithm to Platonic tessellations, we work with their barycentric subdivision so that we have triangulations again. We label the vertices of each simplex in this triangulation such that 0 corresponds to a vertex (which might be ideal), 1 to an edge center, 2 to a face center, and 3 to a center of a Platonic solid (also see Figure~3 of \cite{Goerner:RegTessLinkComps}). Note that a face-pairing in the triangulation always pairs face $i$ with face $i$ such that vertex $j$ goes to vertex $j$. Thus, to specify the triangulation $t$, it is enough to give for each simplex with index $s$ and each face $i$ one index to another simplex. We denote this index by $(t)_{s,i}$ and let $(t)_{s,i}=-1$ when face $i$ of simplex $s$ is unglued. Since the additional gluing permutation that a SnapPy or Regina triangulation stores are not needed, we implement our own much simpler class to store triangulations. Our triangulation class is just an array of simplices where each simplex $s$ is a quadruple $((t)_{s,0}, (t)_{s,1}, (t)_{s,2}, (t)_{s,3})$. If we are interested in orientable manifolds only, we always put the simplices in the array in such a way that all simplices of the same handedness have indices of the same parity, i.e., any two neighboring simplices $s$ and $(t)_{s,i}$ have indices of opposite parity (if $(t)_{s,i}\not = -1$).

\begin{remark}
For the case of tetrahedral manifolds, using the barycentric subdivision instead of SnapPy or Regina triangulations, which encode the gluing permutations, is much slower. Hence, the algorithm described in \cite{FGGTV:tetrahedralCensus} is still relevant.
\end{remark}

A key ingredient in the algorithm described in \cite{FGGTV:tetrahedralCensus} was the usage of the isomorphism signature introduced by Burton in \cite{burton:encode, burton:encode2} to prune the search tree. Recall that the isomorphism signature was a complete invariant of the combinatorial isomorphism type of a triangulation, which can have unglued faces. Since the triangulations used here are fairly special, we can redefine the isomorphism signature to save computation time and, more importantly, memory.

For this, notice that our triangulations are completely determined by their edge-labeled dual 1-skeleton. It is a graph where an edge is labeled by $i$ when it corresponds to pairing face $i$ of one simplex with face $i$ of another simplex. In particular, the edges adjacent to a node have an induced ordering given by $i$. There are well-known deterministic algorithms such as depth-first and breadth-first search \cite{Cormen:Algorithms}, which traverse the nodes of such a graph in an order $n_0, n_1, \dots n_{k-1}$ that only depends on the choice of the start node $n_0$. However, for reasons that become apparent later, we use a different ordering of the nodes here that also only depends on the choice of the start node $n_0$ and is inductively defined as follows: Consider all edges that connect one node among the already ordered ones $n_0, n_1, \dots, n_{j-1}$ with a node different from $n_0, n_1 \dots n_{j-1}$. Among those edges, select only those with lowest label. Among those edges, pick the edge $e$ adjacent to $n_l$ such that $l$ is as low as possible. The next node in the ordering, $n_j$, will be the other node adjacent to $e$.

Given a triangulation $t$ and a choice of start simplex, this gives us a canonical way of (re-)indexing the simplices. If the triangulation $t$ has $k$ simplices, we have $k$ choices of a start simplex and thus obtain a set $S_t$ of $k$ triangulations that are combinatorially isomorphic to $t$. $S_t$ is invariant under combinatorial isomorphisms of $t$. Furthermore, $S_t$ can be ordered because a triangulation $t'\in S_t$ can be encoded by a $k$-tuple of quadruples $((t'_{0,0},t'_{0,1},t'_{0,2},t'_{0,3}),\dots, (t'_{k-1,0},t'_{k-1,1},t'_{k-1,2},t'_{k-1,3}))$ and tuples can be ordered lexicographically. Let $t_0=\min(S_t)$ be the triangulation that comes lexicographically first in $S_t$. Then $t_0$ is canonical in the sense that it is invariant under combinatorial isomorphisms of $t$.

Furthermore, the triangulation $t_0$ has the property that the $(t_0)_{s,0}, (t_0)_{s,1}, (t_0)_{s,2}$ are completely determined by the fixed Platonic solid we use. This is due to our choice of an ordering that exhausts all the simplices of one Platonic solid first before moving on to the next and that always traverses the barycentric subdivision of a Platonic solid in the same way (up to symmetry). Thus, $((t_0)_{0,3},(t_0)_{1,3},\dots,(t_0)_{k-1,3})$ is a complete invariant of the combinatorial isomorphism type of the triangulation $t$, which we call the {\em specialized isomorphism signature}. Since we only use it internally, we do not include a way of serializing this tuple of integers to an ASCII string as Burton does for the ordinary isomorphism signature.

We describe the algorithm to compute the specialized isomorphism signature together with some other basic helpers for triangulations in \algorithmref{algo:barycentricHelpers}.

\begin{algorithm}
\caption{Helpers for barycentric subdivisions of Platonic solids.} \label{algo:barycentricHelpers}

\SetKwProg{Proc}{Procedure}{}{end}%
\SetKwProg{Fn}{Function}{}{end}%
\SetKwFunction{AddPlatonicSolid}{AddPlatonicSolid}%
\SetKwFunction{GlueFaces}{GlueFaces}%
\SetKwFunction{SpecializedIsomorphismSignature}{SpecializedIsomorphismSignature}%
\Fn{\AddPlatonicSolid{Triangulation t, integer p, integer q}}{
\KwResult{Add the barycentric subdivision of the Platonic solid $\{p,q\}$ to \ArgSty{t}. Face 3 of each added simplex will be unglued. Return index of first added simplex.}
Add $\frac{8pq}{2p+2q-pq}$ new simplices to \ArgSty{t}.\\
\tcc{Do next step in such a way that any two new simplices that are neighboring have indices of opposite parity.}
Pair faces 0, 1, 2 of new simplices to form barycentric subdivision of Platonic solid.
\KwRet{index of first added simplex}}
\Fn{\GlueFaces{Triangulation t, integer simp0, integer simp1, integer p}}{
\KwResult{Pair face 3 of the simplices of \ArgSty{t} forming one face of a Platonic solid (with \ArgSty{p}-gons) with those forming another face of another (or possibly the same) Platonic solid such that the simplex \ArgSty{simp0} of \ArgSty{t} is glued to \ArgSty{simp1}. If the simplices \ArgSty{simp0} and \ArgSty{simp1} belong to the same face of the same Platonic solid of \ArgSty{t}, return false.}
\ArgSty{n} $\leftarrow$ 0\\
\While{$(t)_{simp0, 3}=-1$ and $(t)_{simp1, 3}=-1$}{
\If{\ArgSty{simp0} $=$ \ArgSty{simp1}}{
\tcc{Clearly, the two given simplices belong to the same face of the same Platonic solid.}
\KwRet{false}}
\tcc{Pair face 3 of simplex \ArgSty{simp0} and \ArgSty{simp1}.}
$(t)_{simp0,3}\leftarrow simp1, (t)_{simp1,3}\leftarrow simp0$\\
\tcc{For each of the two faces of the Platonic solids, switch to the next simplex of that face by going about the 23-edge.}
\ArgSty{simp0} $\leftarrow$ $(t)_{simp0,0}$ (if $\ArgSty{n}$ even) or $(t)_{simp0,1}$ (otherwise)\\
\ArgSty{simp1} $\leftarrow$ $(t)_{simp1,0}$ (if $\ArgSty{n}$ even) or $(t)_{simp1,1}$ (otherwise) \\
\ArgSty{n} $\leftarrow$ \ArgSty{n} $+1$
}
\tcc{If the two given simplices belonged to the same face of the same Platonic solid, the loop stops early.}
\KwRet{\ArgSty{n} $=$ 2\ArgSty{p}
}}
\Fn{\SpecializedIsomorphismSignature{Triangulation t}}{
\KwResult{The specialized isomorphism signature of \ArgSty{t}.}
For each simplex \ArgSty{s}, obtain a triangulation from \ArgSty{t} by swapping \ArgSty{s} with the first simplex and canonically reindexing all other simplices.\\
\ArgSty{t0} $\leftarrow$ the triangulation that comes lexicographically first among the above.\\
\tcc{Drop the gluing information for face 0, 1, 2}
\KwRet{$((t_0)_{0,3},(t_0)_{1,3},\dots,(t_0)_{k-1,3})$}
}

\end{algorithm}

\subsection{Algorithm}

\begin{algorithm}
\SetKwProg{Proc}{Procedure}{}{end}%
\SetKwProg{Fn}{Function}{}{end}%
\SetKwFunction{FixEdges}{FixEdges}%
\Fn{\FixEdges{Triangulation t, integer p, integer r}}{
\KwResult{\ArgSty{t} is modified in place. Returns ``valid'' or ``invalid''.}
\While{a simplex has an open 01-edge \ArgSty{e} of order 2\ArgSty{r}}{
Let \ArgSty{simp0} and \ArgSty{simp1} be the two simplices adjacent to \ArgSty{e} with unglued face 3.

\If{\FuncSty{GlueFaces(}\ArgSty{t, simp0, simp1, p}\FuncSty{)}}{
\KwRet{``invalid''}
}
}
\KwRet{``valid'' if for each simplex
\begin{itemize}
\item the vertex link of vertex 1 is not a projective plane
\item the order of the 01-edge is $<2r$ (if open) or $=2r$ (if closed)
\end{itemize}
}
}
\caption{Method to check a triangulation.} \label{algo:helper}
\end{algorithm}

\begin{algorithm}
\SetKwProg{Proc}{Procedure}{}{end}%
\SetKwProg{Fn}{Function}{}{end}%
\SetKwFunction{FindAllPlatonicTessellations}{FindAllPlatonicTessellations}%
\SetKwFunction{RecursiveFind}{RecursiveFind}%
\Fn{\FindAllPlatonicTessellations{bool orientable, integer p, integer q, integer r, integer max}}{
\KwResult{Barycentric subdivisions of all (non-)orientable Platonic tessellations $\{p,q,r\}$ up to 
combinatorial isomorphism with at most \ArgSty{max} solids.}
~

result $\leftarrow \{\}$ \tcc*{resulting triangulations}
already\_seen $\leftarrow \{\}$ \tcc*{isomorphism signatures encountered 
earlier}
~

\Proc{\RecursiveFind{Triangulation $t$}}{
\KwResult{Searches all triangulations obtained from $t$ by gluing faces 
or adding Platonic solids.}
~

\tcc{Close 01-edges of order 2r, reject unsuitable triangulations}
\If{ \FuncSty{FixEdges(}$t, r$\FuncSty{)} = ``valid'' }{
\tcc{Skip triangulations already seen earlier}
isoSig $\leftarrow$ \FuncSty{SpecializedIsomorphismSignature(}\ArgSty{t}\FuncSty{)};

\If{ isoSig $\not\in$ already\_seen}{
already\_seen $\leftarrow$ already\_seen $\cup$ \{isoSig\};

\If{ t has no open faces }{
result $\leftarrow$ result $\cup ~\{t\}$\;
}
\Else{
\tcc{This choice results in faster enumeration}
Among all simplices of \ArgSty{t} with unglued face 3, pick one with odd index \ArgSty{simp0} whose edge 01 has order as high as possible.
~

\If{t has less than max $\cdot \frac{8pq}{2p+2q-pq}$ simplices}{
\ArgSty{t1} $\leftarrow$ copy of \ArgSty{t};

\ArgSty{simp1} $\leftarrow$ \FuncSty{AddPlatonicSolid(}\ArgSty{t1, p, q}\FuncSty{)};

\FuncSty{GlueFaces(}\ArgSty{t1, simp0, simp1, p}\FuncSty{);}

\FuncSty{RecurvsiveFind(}\ArgSty{t1}\FuncSty{)};

}

\For{each simplex with index simp1 of t}{
\If{simp1 is even or orientable = false}{
\ArgSty{t1} $\leftarrow$ copy of \ArgSty{t};

\If{\FuncSty{GlueFaces(}\ArgSty{t1, simp0, simp1, p}\FuncSty{)}}{

\FuncSty{RecursiveFind(}\ArgSty{t1}\FuncSty{)};

}
}
}
}
}
}
}
\ArgSty{t} $\leftarrow$ empty triangulation;

\FuncSty{AddPlatonicSolid(}\ArgSty{t, p, q}\FuncSty{)};

\FuncSty{RecursiveFind(}\ArgSty{t}\FuncSty{)};

\If{orientable = false}{
\KwRet{non-orientable triangulations in result}
}
\Else{
\KwRet{result}
}
}
\caption{The main function to enumerate Platonic tessellations.} \label{algo:main}
\end{algorithm}

The algorithm (see \algorithmref{algo:main}) starts with a single Platonic solid and works recursively, at each level picking one open face of a Platonic solid and trying to glue it to any other open face in any configuration or to a new Platonic solid if the given maximal number of solids has not been reached yet. During this search, the same (up to combinatorial isomorphism) complex of Platonic solids will be encountered many times and to avoid duplicate work, we use the specialized isomorphism signature described above.

The algorithm calls into the helper function shown in \algorithmref{algo:helper} to stop recursing if the triangulation does not have the combinatorics suitable to be the barycentric subdivision of a Platonic tessellation of the desired type $\{p,q,r\}$. Note that this function also closes up open edges between vertex 0 and 1 which have the right order (number of adjacent simplices). The recursive search would have closed up that edge by gluing the two open adjacent faces eventually, but doing it in the helper function speeds up the search significantly.

\begin{remark}
Together, the three methods \texttt{AddPlatonicSolid}, \texttt{GlueFaces}, and \texttt{FixEdges} ensure that every edge of every simplex has the right order for tessellations of type $\{p,q,r\}$. \texttt{AddPlatonicSolid}, \texttt{GlueFaces}, and \texttt{FixEdges} also ensure that the link of vertex 3, vertex 2, respectively, vertex 1 is a sphere.\\
Note that the algorithm does not check the vertex link of vertex 0. This is only a problem for the non-orientable closed case where a finite vertex with a projective plane as link would result in non-manifold topology. We can use Regina \cite{regina} to find the ones with non-manifold topology and sort them out later. It turns out that the algorithm only produced non-manifold topology in the closed cubical case.
\end{remark}

\subsection{Multithreading}
This recursive algorithm lacks inherent parallelism, i.e., offers no natural decomposition into elements that can be run concurrently.
Abstractly, the algorithm can be thought of as a search algorithm on the following directed acyclic graph. A node corresponds to an equivalence class $[t]$ of combinatorially isomorphic triangulations. For each node, chose one particular triangulation $t$ from the corresponding equivalence class and add an edge from $[t]$ to $[t1]$ for each $t1$ that was constructed in the \textbf{else} block of the \texttt{RecursiveFind} procedure and successfully processed by \texttt{FixEdges} in \algorithmref{algo:main}.
Starting with the node corresponding to the barycentric subdivision of a single Platonic solid, the algorithm will search all triangulations up to a certain number of simplices and return those ones that have no unglued faces.

Our first parallelization attempts suffered from the problem that all threads but one died quickly leaving almost all the work to the one remaining thread. This is because the above directed acyclic graph is densely connected so threads race for the same nodes.

We eventually decided on a thread pool pattern with a task queue where a task consisted of calling \texttt{RecursiveFind} for some triangulation $t$ and where a task itself could add tasks to the queue. To implement this for \algorithmref{algo:main}, replace the lines ``\texttt{RecursiveFind(}\textit{t1}\texttt{)}'' with code that adds \textit{t1} to the task queue instead. This is not performing well yet, and we added another optimization: we replaced the lines ``\texttt{RecursiveFind(}\textit{t1}\texttt{)}'' instead with code that adds \textit{t1} to the task queue if there are idle threads and otherwise continues to call \texttt{RecursiveFind(}\textit{t1}\texttt{)}. 

\textit{result} and \textit{already\_seen} will be shared among the threads and must be guarded by mutexes. In particular, the test \textit{isoSig}$\not\in$\textit{already\_seen} and the following instruction of adding \textit{isoSig} to \textit{already\_seen} need to be one atomic operation.

\subsection{Implementation}

We implemented \algorithmref{algo:main} in C++.  We used the boost library \cite{boost} to implement multithreading. The multithreaded implementation was successful and resulted in about a 10 times speed-up compared to the single-threaded implementation on a 12 core Xeon E5-2630. We also used the Regina library, though only to convert our triangulation objects into isomorphism signatures that can be understood by Regina or SnapPy.

\section{The enumeration of Platonic manifolds} \label{sec:enumPlatMan}

\subsection{Isometry signature for cusped manifolds} \label{sec:isometrySig}

The census of Platonic manifolds is obtained from the census of Platonic tessellations by grouping the tessellations by isometry type. We do this by grouping them by their isometry signature.

Recall that the isometry signature introduced in \cite{FGGTV:tetrahedralCensus} is a complete invariant of a cusped hyperbolic 3-manifold based on the canonical cell decomposition introduced by Epstein and Penner \cite{EP} (also see \cite[Definition~3.1]{FGGTV:tetrahedralCensus}). If the canonical cell decomposition contains non-tetrahedral cells, there is a canonical way of turning it into a triangulation called the {\em canonical retriangulation} (see \cite[Definition~3.3]{FGGTV:tetrahedralCensus}). Thus, we always obtain a triangulation that is canonical and we can compute its isomorphism signature, which was defined by Burton \cite{burton:encode, burton:encode2}. We call the result the {\em isometry signature}. The canonical retriangulation and isometry signature can be computed in SnapPy, version 2.3.2 or later, as follows:
\begin{verbatim}
>>> M=Manifold("m137")
>>> M.isometry_signature()
'sLLvwzvQPAQPQccghmiljkpmqnoorqrrqfafaoaqoofaoooqqaf'
>>> T = M.canonical_triangulation()
\end{verbatim}

For the above computations, the SnapPea kernel of SnapPy uses numerical methods, which are not verified and could potentially wrong results. If SnapPy is used inside Sage, we can give \texttt{verified=True} as extra argument to use methods that instead are proven to give either the correct result or no result:
\begin{verbatim}
>>> M=Manifold("m137")
>>> M.isometry_signature(verified = True)
'sLLvwzvQPAQPQccghmiljkpmqnoorqrrqfafaoaqoofaoooqqaf'
>>> T = M.canonical_triangulation(verified = True)
>>> len(T.isomorphisms_to(T)) # The verified size of the symmetry group of M
2
\end{verbatim}

Verifying the canonical cell decomposition when all cells are tetrahedral was already described in Dunfield, Hoffman, Licata~\cite{DHL} using HIKMOT \cite{hikmot}. In \cite{FGGTV:tetrahedralCensus}, we described how to verify a canonical cell decomposition that might have non-tetrahedral cells for cusped arithmetic manifolds with known trace field. This, however, does not cover the cusped dodecahedral manifolds, which are non-arithmetic.

The author has generalized the algorithm for verified canonical cell to any cusped hyperbolic manifold and contributed it to SnapPy. The implementation first tries to use interval arithmetic methods to verify the canonical cell decomposition. Interval arithmetic methods can prove inequalities but cannot prove equalities, thus they can only verify canonical cell decompositions with tetrahedral cells. Therefore, if the verification with interval arithmetic failed, the algorithm tries exact methods next.

We refer the reader to the ``Verified computations'' section of the SnapPy documentation \cite{SnapPy} for more examples and plan a future publication to explain the underlying math in depth.

\subsection{Invariant for closed Platonic manifolds} \label{sec:closedInvariant}

We regard two covering spaces $\tilde{M}\to M$ and $\tilde{M'}\to M$ equivalent if there is an isomorphism $\tilde{M}\to\tilde{M'}$ commuting with the covering maps. Given a manifold $M$ and a natural number $n>0$, let $C_n(M)$ be the multiset of pairs $(\mathrm{type}(\tilde{M}\to M), H_1(\tilde{M}))$ where $\tilde{M}\to M$ is a connected covering space of degree $n$ and where $\mathrm{type}(\tilde{M}\to M)$ takes the values ``cyclic'', ``regular'', and ``irregular'' based on the covering type. $C_n(M)$ can be computed with SnapPy and is an invariant of $M$.

For example, SnapPy's census database uses a manifold hash for faster lookups that is computed from $C^\mathrm{SnapPy}(M)=(H_1(M), C_2(M), C_3(M))$. We used $C^\mathrm{SnapPy}(M)$ to start separating the Platonic tessellations. However, we were left with cases where this invariant could not tell apart several manifolds for which SnapPy could not find an isomorphism between them either. For these cases, we used $C_n(M)$ or $C^\mathrm{cyclic}_n$ (where we consider only cyclic covers of degree $n$) with higher $n$ to resolve the situation.

It is prohibitively expensive to compute these higher $C_n(M)$ or $C^\mathrm{cyclic}_n(M)$ for all closed Platonic tessellations in the census. Yet, for simplicity, we want to define a single invariant strong enough to separate all closed Platonic manifolds in the census. We thus came up with the following expression, which is rather engineered for this purpose than canonical, but still an invariant:
$$C^\mathrm{proprietary}(M) = \begin{cases}
(C^\mathrm{SnapPy}(M),C^\mathrm{cyclic}_5(M)) & \mbox{if } C^\mathrm{SnapPy}(M) = (\{(\mathrm{cyclic},(\BZ/5)^3)\},\{\},\{\}) \\
(C^\mathrm{SnapPy}(M),C_6(M)) & \mbox{if } C^\mathrm{SnapPy}(M) = (\{(\mathrm{cyclic},\BZ/29)\},\{\},\{\}) \\
\qquad \vdots & \mbox{five more special cases}\\
C^\mathrm{SnapPy}(M) & \mbox{otherwise}
 \end{cases}$$

\section{The census} \label{sec:results}

We ran the multithreaded algorithm in Section~\ref{sec:algo} to create the census of hyperbolic Platonic tessellations and manifolds on a 12 core Xeon E5-2630 with 128 Gb memory. For each case, we picked the highest number of Platonic solids so that the algorithm would still finish within a couple of days and without running out of memory. We then grouped the tessellations by isometry type using the invariants described in Section~\ref{sec:enumPlatMan} and converted the result to a SnapPy census or Regina \cite{regina} file. Since we obtained over a million tessellations, computing the verified isometry signatures in the last step also took several days of computation time\footnote{Even more time was needed to compute $C^\mathrm{SnapPy}(M)$ which SnapPy hashes for faster look-up.}.

The results are shown in Theorem~\ref{thm:Main} and Table~\ref{table:census336}, \ref{table:census344}, \ref{table:census436}, \ref{table:census536}, \ref{table:census353}, \ref{table:census435}, and \ref{table:census535}.

\subsection{Naming}

We give hyperbolic Platonic manifolds names as follows\footnote{This differs slightly from the names introduced in \cite{FGGTV:tetrahedralCensus} in that we add one more leading zero for consistency across all manifolds of the small census.}:
$$ \underbrace{\texttt{o}}_{\substack{\text{orientability:}\\ "\texttt{o}"\\ "\texttt{n}"}} \underbrace{\texttt{dode}}_{\substack{\text{solid:}\\ "\texttt{tet}" \\ "\texttt{cube}" \\ "\texttt{oct}" \\ "\texttt{dode}" \\ "\texttt{ico}"}} \underbrace{\texttt{cld}}_{\substack{\text{closed:}\\ "\texttt{cld}"\\ \text{cusped:} \\ ""}} \underbrace{\texttt{03}}_{\substack{\text{number}\\ \text{of} \\ \text{solids}}} \texttt{\_} \underbrace{\texttt{00027}}_{\text{Index}}.$$
The different Platonic tessellations corresponding to the same manifold are named with an additional index, e.g., \texttt{ododecld03\_00027\#0}, \texttt{ododecld03\_00027\#1}. The indices are chosen deterministically using the lexicographic order on the isomorphism signature of a tetrahedral Platonic tessellation, respectively, the barycentric subdivision of a non-tetrahedral Platonic tessellation, similarly to \cite[Section 4.1]{FGGTV:tetrahedralCensus}.

\subsection{SnapPy census} \label{sec:SnapPyCensus}

The small census of hyperbolic Platonic manifolds is already available in a SnapPy installation, beginning with version 2.4. It can be used just like any other census in SnapPy, for example:
\begin{verbatim}
>>> M = Manifold("odode01_00001") # only works for small census
>>> M = DodecahedralOrientableCuspedCensus["odode01_00001"]
>>> len(OctahedralOrientableCuspedCensus(solids=3)) # Number mfds with 3 octs
24
>>> M = Manifold("x101")
>>> CubicalNonorientableCuspedCensus.identify(M)
ncube01_00004(0,0)
\end{verbatim}

The large census of hyperbolic Platonic manifolds needs to be obtained from \cite{goerner:platonicCensusData} first and imported into SnapPy (using ``\texttt{from platonicCensus import *}'' in the \texttt{snappy} directory which contains \texttt{platonicCensus.py}, also see \texttt{README.txt}) before it can be used just as the examples above except for the first line.

\begin{remark}
Similar to the SnapPy \texttt{OrientableClosedCensus}, closed manifolds in the Platonic census are given as Dehn-fillings on a 1-cusped manifold. SnapPy can automatically convert a triangulation with finite vertices into this form. However, we sometimes had to modify the triangulation to ensure that SnapPy can find a geometric solution to the gluing equations.
\end{remark}

\subsection{Regina files}

We provide the census of hyperbolic Platonic tessellations as Regina files at \cite{goerner:platonicCensusData}. Each Regina file contains the cusped or closed tessellations for one Platonic solid and is structured into a three-level hierarchy as follows:
\begin{itemize}
\item Container nodes, each for a different number of solids.
\item Container nodes, each for one hyperbolic Platonic manifold, i.e., it groups all tessellations that are isometric as manifolds and is named after the manifold.
\item Triangulation nodes, each containing
\begin{itemize}
\item a Platonic tessellation (in tetrahedral case) or its barycentric subdivision\footnote{The triangulation in the regina file is combinatorially isomorphic to the barycentric subdivision, but the vertices might not be indexed as in Section~\ref{sec:barySubdiv}. This is because the isomorphism signature was used in the intermediate steps. The method \texttt{conform\_vertex\_order} in \texttt{tools/conform.py} can be used to reorder the vertices to follow the convention again.} or
\item the canonical retriangulation of the corresponding manifold (in cusped case)
\end{itemize}
\end{itemize}

\section{Properties of Platonic tessellations and manifolds}

\label{sec:propPlat}

In this section, we discuss and give some properties of the tessellations and manifolds in the Platonic census.

\subsection{Closed tessellations and the Seifert-Weber space} \label{sec:seifertWeber}

\begin{table}
\caption{Orientable closed Platonic tessellations with one dodecahedron. Every tessellation in the list is chiral.\label{tbl:closedWithOneDode}}

\bigskip

\begin{tabular}{l|c|c|c}
\multicolumn{1}{c|}{Tessellation} & self-dual &regular & $H_1$ \\ \hline \hline
\texttt{ododecld01\_00000} & Yes & No  & $\BZ/35$ \\ \hline
\texttt{ododecld01\_00001} & No & No  & $\BZ/48$ \\ \hline
\texttt{ododecld01\_00002} & Yes & No  & $\BZ/29$ \\ \hline
\texttt{ododecld01\_00003} & Yes & No  & $(\BZ/15)^2$ \\ \hline
\texttt{ododecld01\_00004} & No & No  & $(\BZ/5)^3$ \\ \hline
\texttt{ododecld01\_00005} & No & No  & $(\BZ/3)^2$ \\ \hline
\texttt{ododecld01\_00006} & Yes & No  & $\BZ/5\oplus\BZ/15$ \\ \hline
\texttt{ododecld01\_00007\#0} & Yes & Yes  & $(\BZ/5)^3$ \\
\texttt{ododecld01\_00007\#1} & Yes & No  & $(\BZ/5)^3$
\end{tabular}
\end{table}

The model regular tessellations $\{3,5,3\}$ and $\{5,3,5\}$ are self-dual. This means that the dual $T^*$ of a Platonic tessellation $T$ of type $\{3,5,3\}$ or $\{5,3,5\}$ is of the same type but might or might not be combinatorially isomorphic to $T$. If $T^*$ and $T$ are combinatorially isomorphic, we say that $T$ is self-dual.

Table~\ref{tbl:closedWithOneDode} lists all closed tessellations with one dodecahedron. Note that \texttt{ododecld01\_00004} and \texttt{ododecld01\_00007} form the only pair of manifolds in the table which cannot be distinguished by their first homology groups (they can be distinguished by the homology groups of their 5 fold covers). 

\texttt{ododecld01\_00007} is actually the Seifert-Weber space \cite{seifertWeber:ref}. Recall that the Seifert-Weber space is the hyperbolic manifold obtained by taking a regular hyperbolic dodecahedron $P$ with dihedral angle $2\pi/5$ and gluing opposite face by a $3\pi/5$ rotation. Consider the Dirichlet domains obtained by picking as base point the center of $P$ itself or a face center, edge center, or vertex of $P$. Each of these Dirichlet domains turns out to be a regular dodecahedron again with the same dihedral angle. Thus each of these choices gives a Platonic tessellation of the Seifert-Weber space. Picking the center of $P$ as base point gives the tessellation $T$ of the Seifert-Weber space by $P$ itself. The dual of $T^*$ is obtained when picking any vertex of $P$ as base point. $T$ and $T^*$ are combinatorially isomorphic (and denoted by \texttt{ododecld01\_00007\#0}). In fact, there are orientation-reversing isometries of the Seifert-Weber space that take $T$ to $T^*$ (namely, any isometry corresponding to the reflection about the bisecting plane of the center of $P$ and a vertex of $P$). Note that while the Seifert-Weber space is amphichiral, $T$ and $T^*$ are actual chiral as tessellations (i.e., they have no orientation-reversing combinatorial automorphism). $T$ and $T^*$ are actually hiding the orientation-reversing symmetries of the Seifert-Weber space. The symmetries of $T$ or $T^*$ (which are both regular tessellations) together with any orientation-reversing symmetry generate the full symmetry group of the Seifert-Weber space, which is $S_5$. In other words, all symmetries of the Seifert-Weber space occur as symmetries of the triangulation obtained from $T$ or $T^*$ by barycentric subdivision.\\
All Platonic tessellations obtained from Dirichlet domains with base point being a face or edge center are combinatorially isomorphic (and denoted by \texttt{ododecld01\_00007\#1}).

For the remaining type $\{4,3,5\}$ of closed Platonic tessellations, we have the following lemma.

\begin{lemma}
The number of cubes of a closed cubical tessellation $\{4,3,5\}$ is a multiple of 5.
\end{lemma}

\begin{proof}
Such a tessellation corresponds to a torsion-free subgroup $\Gamma$ of $\Gamma_{\{4,3,5\}}$, which has torsion elements of order $5$. Thus, the index of $\Gamma$ in $\Gamma_{\{4,3,5\}}$ must be a multiple of $5$. However, the number of fundamental domains of $\Gamma_{\{4,3,5\}}$ in a cube is 48 and thus coprime to 5.
\end{proof}

\subsection{Hidden symmetries and isometric tessellations} \label{sec:hiddenSym}

The only non-arithmetic symmetry group among the hyperbolic tessellations in Table~\ref{tbl:modelRegTess} is $\Gamma_{\{5,3,6\}}$ (see \cite[Section 13.1, 13.2]{MR}). As explained in \cite{neumannReid:Topology90Arithmetic}, Margulis Theorem (see, e.g., \cite[Theorem~10.3.5]{MR}) thus implies that the commensurability class of $\Gamma_{\{5,3,6\}}$ has a maximal element, namely the commensurator of $\Gamma_{\{5,3,6\}}$ given by (also see \cite{Walsh})
$$
\Comm(\Gamma)=\left\{g\in \Isom(\BH^3) ~\big|~ \left[\Gamma:\Gamma\cap g\Gamma g^{-1}\right]<\infty,~ \left[g\Gamma g^{-1}:\Gamma\cap g\Gamma g^{-1}\right]<\infty \right\},
$$
i.e., the subgroup of those elements $g\in \Isom(\BH^3)$ such that $\Gamma$ and $g\Gamma g^{-1}$ are commensurable.

This maximal element is actually $\Gamma_{\{5,3,6\}}$ itself, which is also equal to its own normalizer. In other words, $\Gamma_{\{5,3,6\}}$ admits neither symmetries nor hidden symmetries. This fact implies that the related tessellations cannot hide symmetries:

\begin{lemma}
Every manifold commensurable with the orbifold $\BH^3/\Gamma_{\{5,3,6\}}$ is a covering space of the orbifold and thus a cusped dodecahedral manifold. Every cusped dodecahedral manifold $M$ has a unique Platonic tessellation. Furthermore, no cusped dodecahedral tessellation hides symmetries {\rm (}in the sense of Definition~\ref{def:hiddenSymTess}{\rm\/).}
\end{lemma}

\begin{proof}
Given a manifold $M$ commensurable with $\BH^3/\Gamma_{\{5,3,6\}}$, we obtain a dodecahedral tessellation on $M$ (induced from the model regular tessellation $\{5,3,6\}$) by choosing a $\Gamma\subset\Gamma_{\{5,3,6\}}$ such that $M\cong\BH^3/\Gamma$. Two such choices of $\Gamma$ differ by conjugation by $g$. Since both $\Gamma$ and $g\Gamma g^{-1}$ are finite index subgroups of $\Gamma_{\{5,3,6\}}$, they are commensurable and thus $g\in\Comm(\Gamma_{\{5,3,6\}})=\Gamma_{\{5,3,6\}}$. Hence, two such choices yield the same tessellation. Furthermore, a cusped dodecahedral manifold never admits a non-dodecahedral Platonic tessellation since $\Gamma_{\{5,3,6\}}$ is not commensurable with any other symmetry group in in Table~\ref{tbl:modelRegTess}. Similarly, a symmetry of $M$ corresponds to an element $g\in \Isom(\BH^3)$ such that $\Gamma=g\Gamma g^{-1}$ and thus again $g\in\Comm(\Gamma)=\Comm(\Gamma_{\{5,3,6\}})=\Gamma_{\{5,3,6\}}$, so the symmetry is not hidden by the tessellation.
\end{proof}

This is in contrast to all other tessellation types $\{p,q,r\}$ in Table~\ref{tbl:modelRegTess} where Margulis Theorem says that the commensurator of the symmetry group $\Gamma_{\{p,q,r\}}$ is dense in $\Isom(\BH^3)$ since they are arithmetic.
Thus, we expect examples of Platonic manifolds with non-unique tessellations and symmetries hidden by tessellations. An example is \texttt{otet10\_00027\#0}, see Section~\ref{section:tessTools}.
Note that these examples have at least two cusps since Lemma~5.15 in \cite{FGGTV:tetrahedralCensus} generalizes to cusped Platonic and to the closed Platonic tessellations of non-self dual type.

\subsection{Cusped cubical and tetrahedral tessellations}

Exactly two symmetry groups in Table~\ref{tbl:modelRegTess} are commensurable, namely $\{3,3,6\}$ and $\{4,3,6\}$ (also see \cite{neumannReid:SmallVolOrbs}). More explicitly, an ideal regular cube can be subdivided into five regular ideal tetrahedra (see, e.g., \cite[Figure~1]{FGGTV:tetrahedralCensus}) introducing a new diagonal on each face of the cube. Given an ideal cubical tessellation, we can subdivide each cube into regular tetrahedra individually and obtain a tetrahedral tessellation if the choices of the newly introduced diagonals are compatible with the face pairing of the cubes. There is either no way or exactly two ways of subdividing a cusped cubical tessellation into a tetrahedral tessellation. These correspond to two-colorings of the 1-skeleton of the cubical tessellation regarded as a graph where vertices correspond to cusps and edges to edges. Namely, fix a color and draw a diagonal on each cubical face between vertices of the that color to obtain the subdivision. In particular, a 1-cusped cubical tessellation cannot be divided into a tetrahedral one.

\begin{example}
\texttt{ocube01\_00001\#0} is a cubical tessellation that cannot be subdivided into a tetrahedral tessellation. \texttt{ocube02\_00026\#0} and \texttt{ocube02\_00027\#0} can both be subdivided into tetrahedral tessellations. The two possible choices of coloring yield two combinatorially non-isomorphic tetrahedral tessellations for \texttt{ocube02\_00026\#0} but only one tetrahedral tessellation up to combinatorial isomorphism for \texttt{ocube02\_00027\#0}.
\end{example}

\subsection{Regular tessellations}

Table~\ref{table:RegTess} lists all regular tessellations in the census as defined in Definition~\ref{def:RegTess} and compares them with the characterization given in \cite{Goerner:RegTessLinkComps} which classified regular tessellations with small cusped modulus.

\begin{table} \label{table:RegTess}
\caption{Regular tessellations in the census of Platonic tessellations. See \cite{Goerner:RegTessLinkComps} for notation.}

\bigskip

\begin{center}
\begin{tabular}{l|c}
Platonic census name & Other name \\ \hline \hline
\texttt{otet10\_00027} & $\MunivPrinCong{\{3,3,6\}}{2}$ \\ \hline
\texttt{ooct04\_00042\#1} & $\MunivPrinCong{\{4,3,6\}}{2}$ \\ \hline
\texttt{ooct05\_00059\#1} & $\MunivPrinCong{\{4,3,6\}}{2+i}$ \\ \hline
\texttt{ooct08\_354962\#2} & $\mathcal{Z}\in \mathcal{C}^{\{3,4,4\}}_{2+2i}$ \\ \hline
\texttt{ocube02\_00042} & $\mathcal{Z}_0\in \mathcal{C}^{\{4,3,6\}}_2$ \\ \hline
\texttt{ocube06\_09263} & - \\ \hline
\texttt{ocube06\_03577\#1} & $\MunivPrinCong{\{4,3,6\}}{1+\zeta}$ \\ \hline
\texttt{odode02\_00912} & $\mathcal{Z}_0\in \mathcal{C}^{\{5,3,6\}}_2$ \\ \hline \hline
\texttt{ododecld01\_00007\#0} & Seifert-Weber space 
\end{tabular}
\end{center}
\end{table}

\subsection{Tools for further investigations} \label{section:tessTools}

We implemented various methods to check whether a given Platonic tessellation has the properties described earlier, see \cite[tools/]{goerner:platonicCensusData}. We provide a script that can be used from the shell and summarizes these properties. Here is an example of its usage:

\begin{verbatim}
$ python tools/showProperties.py otet10_00027
Properties of otet10_00027   (isometric to ocube02_00025)
  Number of tessellations: 2
  otet10_00027#0: self_dual - regular -   chiral - hidesSyms YES (48/240)
                  (coarsens to ocube02_00025#0)
  otet10_00027#1: self_dual - regular YES chiral - hidesSyms -   (   240)
                  (coarsens to ocube02_00025#0)
\end{verbatim}

It shows that the Platonic manifold $\texttt{otet10\_00027}$ has two combinatorially non-isomorphic tetrahedral tessellations. Both are actually obtained from the two different choices when subdividing the cubes of the tessellation $\texttt{ocube02\_00025\#0}$ into tetrahedra. The first tetrahedral tessellation hides symmetries (it has 48 combinatorial automorphisms but the underlying tetrahedral manifold has 240 isometries). The second tetrahedral tessellation has no hidden symmetries, in fact, it is a regular tessellation, namely the complement of the minimally twisted 5-component chain link as described in \cite{dunfield:virthaken}.

\section{Augmented knotted trivalent graphs} \label{sec:AugKTG}

Introduced in \cite{Roland:AKTG}, all augmented knotted trivalent graphs (AugKTG) are obtained from the complete, planar graph of 4 vertices by applying A-, U-, and X-moves. An A-move replaces a trivalent vertex by a triangle and a U-, respectively, X-move unzips an edge between two distinct trivalent vertices while adding an unknotted component about the edge and optionally introducing a half-twist (X-move). Figure~\ref{fig:AugKTGExample} shows an example. Without loss of generality, we can assume that all A-moves are applied before any U- or X-move. Here, we assume that $n$ A-moves are always followed by $n+2$ U- and X-moves so that the resulting AugKTG is a link, whose complement is Platonic and can be tessellated by $2(n+1)$ octahedra \cite[Lemma~3]{Roland:AKTG}. Recall that links are in general not determined by their complement and we list only one AugKTG for each class of AugKTGs with homeomorphic complement.

We have enumerated AugKTGs up to 6 (for the small census), respectively 8 (for the large census) octahedra, see Table~\ref{table:census344}. The diagrams can be found in \cite[AugKTG/diagrams]{goerner:platonicCensusData}. They are also shipped with the census (comments from Section~\ref{sec:SnapPyCensus} apply) and can be accessed as follows:

\begin{verbatim}
>>> AugKTGs = OctahedralOrientableCuspedCensus(isAugKTG=True)
>>> for M in AugKTGs[:10]: # For first 10
...    print M.DT_code() # Show DT code
...    M.plink() # And link diagram
[(10, -20), (2, 26, -16), (-4, 24, -6, -22, 14), (8, 12, -18)]
       ...
\end{verbatim}

\begin{figure}
\includegraphics[height=7cm]{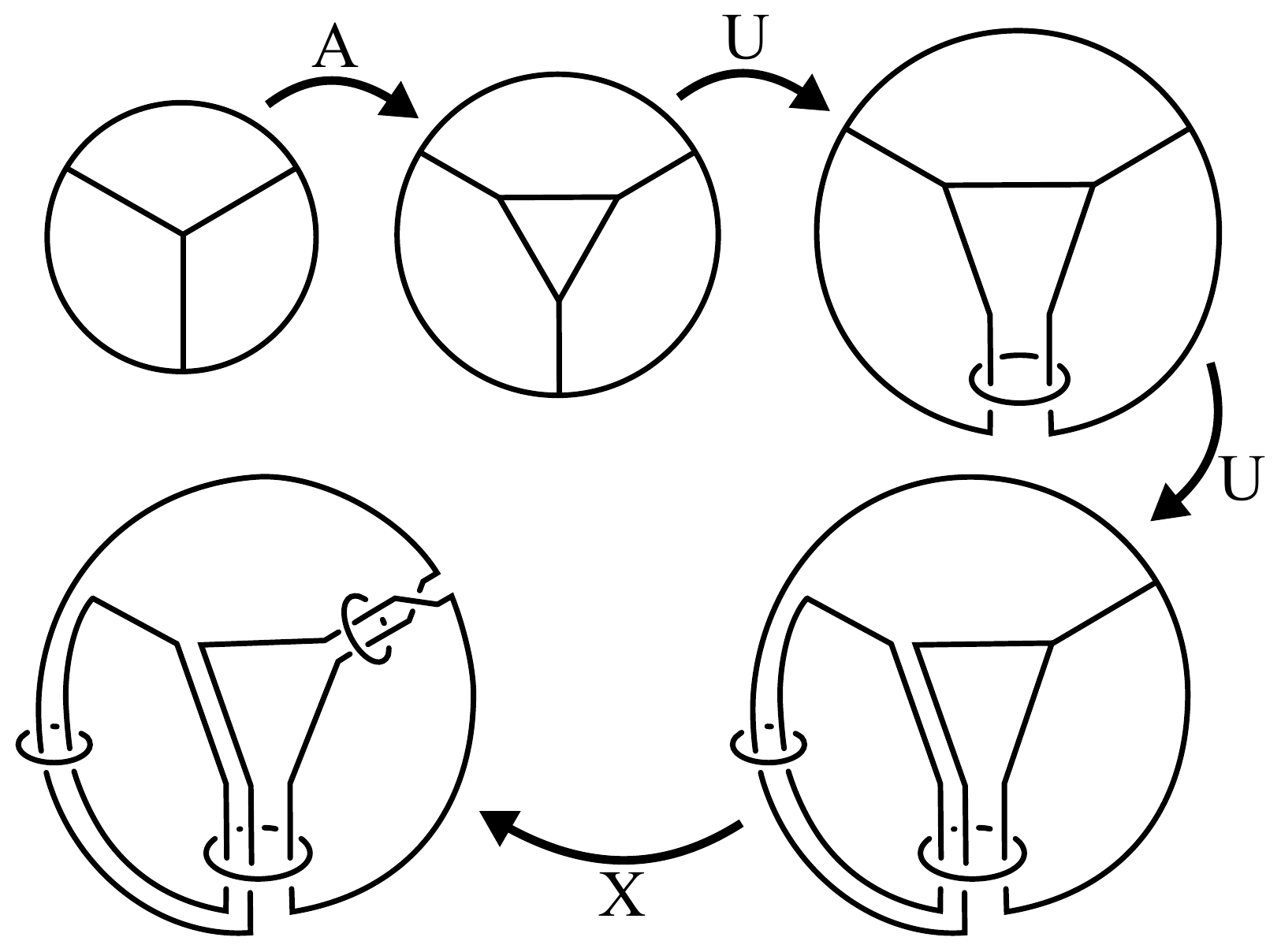}
\caption{Construction of an augmented knotted trivalent graph.\label{fig:AugKTGExample}}
\end{figure}

We now describe the algorithm for enumerating AugKTGs, which we implemented in C++ and boost.

\subsection{Presentation of AugKTGs}

We encode a planar projection of an AugKTG as a fat graph with an extra flag for each half-edge to indicate under-crossings. A fat graph is a graph together with a cyclic ordering of the half-edges adjacent to a vertex for each vertex. We use half-edges to encode a fat graph where we store for each half-edge $h$
\begin{enumerate}
\item a pointer to $h_{\mathrm{other}}$, the other half edge that together with $h$ forms an edge
\item a pointer to $h_{\mathrm{next}}$, the half-edge adjacent to the same vertex as $h$ and next in the cyclic ordering.
\end{enumerate}

In case of an AugKTG, each vertex of the fat graph is either trivalent or quadvalent to indicate a crossing where for two opposite half-edges the extra flag is set to indicate that they are crossing under the other two half-edges.

\subsection{Isomorphism signature for AugKTGs}

We can apply the ideas behind Burton's isomorphism signature for 3-dimensional triangulations \cite{burton:encode, burton:encode2} to fat graphs to obtain a complete invariant of a fat graph up to fat graph isomorphism. Given a fat graph encoded as above, there is a deterministic algorithm similar to the one in Section~\ref{sec:specialIsoSig} to traverse the half-edges in an order $h_0, h_1, \dots, h_{k-1}$ that only depends on the choice of a start half-edge $h_0$. We chose an algorithm so that $h_{2i+1}=(h_{2i})_{\mathrm{other}}$. Once, we have (re-)indexed the half-edges in the traversal order, we can encode the fat graph as a tuple by taking for each half edge the index of the next half edge. And, given a fat graph, we can similarly to Section~\ref{sec:specialIsoSig} pick the lexicographically smallest tuple among all the tuples obtained from different choices of start half-edges. This gives as an isomorphism signature for fat graphs.

If we add to the tuple each half edge's extra flag to indicate under-crossings, we have an isomorphism signature for planar projections of AugKTGs. For speed, we added extra code to process a planar projection of an AugKTG before computing its isomorphism signature such that we obtain the same result when
\begin{enumerate}
\item flipping all crossings of a planar projection of an AugKTG
\item swapping all crossings of a ``belt'' - an unknotted circle that splits into two parts with one part only under-crossings and one part only over-crossings.
\end{enumerate}  

\subsection{Enumeration}

The algorithm to enumerate all AugKTGs takes as input the number of A-moves. It then recursively performs first all possible A-moves and then all possible U- and X-moves. Many different sequences of these moves can yield to the same AugKTG. To reduce the re-enumeration of the same AugKTG, we keep a set of isomorphism signatures of AugKTGs and stop recursing if we encounter an isomorphism signature that was already added earlier. In the recursion, we also simplify the projection of the AugKTG by applying Reidemeister I moves when possible.

The program emits the resulting links as PD codes and we identify their complements in the octahedral census.

\section{Potential applications}

We say that a hyperbolic 3-manifold $M$ bounds geometrically if $M$ is the totally geodesic boundary of a complete, finite volume hyperbolic 4-manifold. Recent work has shown interesting connections between Platonic manifolds and 3-manifolds bounding geometrically. For example, Martelli \cite{martelli:boundingGeometrically} shows that octahedral orientable cusped and cubical orientable closed (dual to right-angled closed dodecahedral) manifolds bound geometrically. In case of tetrahedral cusped manifolds, this is not known in general but Slavich \cite{slavich:figEightBoundingGeometrically} gives a construction of a hyperbolic 4-manifold bounded geometrically by the complement of the figure-eight knot.

Both the censuses and the techniques given in this paper might be useful for further investigations. In particular, the techniques could be generalized to 4-dimensional locally regular tessellations (perhaps using the upcoming extension of Regina to 4-manifolds), which form the basis of the constructions in \cite{martelli:boundingGeometrically, slavich:figEightBoundingGeometrically}.

\subsection*{Acknowledgements}
The author gratefully thanks Stavros Garoufalidis for many helpful conversations. The author was supported in part by a National Science Foundation grant DMS-11-07452.

\clearpage

\section*[Appendix]{Appendix\except{toc}{: Hyperbolic ideal cubulations can be subdivided\\ into ideal geometric triangulations}}

\begin{abstract}
Consider a cusped hyperbolic 3-manifold that can be decomposed into (not necessarily regular) ideal convex cubes. We prove that the cubes can be subdivided into non-flat ideal tetrahedra in such a way that they form an ideal geometric triangulation.
\end{abstract}

\maketitle

\renewcommand{\thetheorem}{\arabic{theorem}}

\begin{theorem}[Main theorem]
An ideal cubical tessellation, or more generally, a cell decomposition of a hyperbolic 3-manifold into ideal geodesic convex cubes can be subdivided into an ideal geometric triangulation.
\end{theorem}

\begin{lemma}
Consider an ideal convex cube with a choice of one of the two diagonals for each face. These diagonals come from a subdivision of the cube into non-flat ideal tetrahedra if the cube has a vertex $v$ adjacent to three of the chosen diagonals.
\end{lemma}

\begin{proof}
Consider the 2-cell complex obtained by subdividing the cube's surface along the given diagonals and remove all cells that are adjacent to $v$. Coning this 2-cell complex to $v$ yields a subdivision of the cube.
\end{proof}

\begin{remark}
This construction was inspired by Lou, Schleimer, Tillmann \cite{LuoSchleimerTillmann:virtualGeometricTrig} and generalizes to any convex polyhedra $P$ with a choice of non-intersecting face diagonals subdividing each face into triangles: $P$ can be subdivided into non-flat ideal tetrahedra compatible with the given choice of diagonals if $P$ has a vertex $v$ such that for each face $f$ adjacent to $v$, each diagonal on $f$ is also adjacent to $v$.
\end{remark}

\begin{remark}
A case by case analysis for the cube actually reveals that a choice of diagonals comes from a subdivision of the cube if and only if the cube has a vertex adjacent to either three or none of the chosen diagonals.
\end{remark}

\begin{proof}[Proof of main theorem]
We call a sequence $f_0, f_1, f_2, \dots, f_{k-1}$ of distinct faces of the cubulation a face cycle if $f_i$ and $f_{i+1}$ are opposite faces of the same cube for each $i=0, \dots, {k-1}$ (indexing is cyclic so $f_0=f_k$). Note that the reverse $f_{k-1}, f_{k-2}, \dots, f_0$ is also a face cycle. Thus, the faces of a cubulation naturally partition into unoriented face cycles, but we can fix a choice of orientation for each face cycle.\\
Recall that each face of the cubulation corresponds to two faces of two (not necessarily distinct) cubes. Orienting the face cycles gives a canonical way to pick one of those two faces for each face of the cubulation. For each cube and each pair of opposite faces of that cube, one of the two faces will be picked - even if the face cycle runs through the cube multiple (up to 3) times.\\
The three faces picked from a cube will thus meet at a vertex. Pick the diagonals of those three faces so that they meet at that vertex. By the above Lemma, this choice of diagonals allows the cubes to be subdivided.
\end{proof}

\clearpage



\bibliographystyle{hamsalpha}
\bibliography{platonicCensusBib}

\newcommand{\etalchar}[1]{$^{#1}$}
\providecommand{\bysame}{\leavevmode\hbox to3em{\hrulefill}\thinspace}
\providecommand{\href}[2]{#2}
\providecommand{\eprint}{\begingroup \urlstyle{rm}\Url}
\begin{thebibliography}{FGG+16}

\bibitem[Ago]{Agol:CongruenceLink}
Ian Agol, \emph{Thurston's congruence link},
  \url{http://citeseerx.ist.psu.edu/viewdoc/download?doi=10.1.1.121.3407&rep=rep1&type=pdf}.

\bibitem[AR92]{AR:dodecahedral}
I.~R. Aitchison and J.~H. Rubinstein, \emph{Combinatorial cubings, cusps, and
  the dodecahedral knots}, Topology '90 ({C}olumbus, {OH}, 1990), Ohio State
  Univ. Math. Res. Inst. Publ., vol.~1, de Gruyter, Berlin, 1992, pp.~17--26.

\bibitem[Bak02]{Baker:AllLinksAreSublinksOfArithmeticLinks}
Mark~D. Baker, \emph{All links are sublinks of arithmetic links}, Pacific J.
  Math. \textbf{203} (2002), no.~2, 257--263.

\bibitem[BR14]{bakerReid:prinCong}
Mark~D. Baker and Alan~W. Reid, \emph{Principal congruence link complements},
  Ann. Fac. Sci. Toulouse Math. (6) \textbf{23} (2014), no.~5, 1063--1092.

\bibitem[Bur]{regina}
Benjamin Burton, \emph{{R}egina, software for $3$-manifold topology and normal
  surface theory}, \url{https://regina-normal.github.io/} (30/12/2015).

\bibitem[Bur11a]{burton:encode}
Benjamin~A. Burton, \emph{Simplification paths in the {P}achner graphs of
  closed orientable 3-manifold triangulations}, 2011, \eprint{arXiv:1110.6080},
  Preprint.

\bibitem[{Bur}11b]{burton:encode2}
Benjamin~A. {Burton}, \emph{{The Pachner graph and the simplification of
  3-sphere triangulations.}}, {Proceedings of the 27th annual symposium on
  computational geometry, SoCG 2011, Paris, France, June 13--15, 2011}, New
  York, NY: Association for Computing Machinery (ACM), 2011, pp.~153--162
  (English).

\bibitem[CD01]{conder:regmaps}
Marston Conder and Peter Dobcs\'{a}nyi, \emph{Determination of all regular maps
  of small genus}, Journal of Combinatorial Theory, Series B \textbf{81}
  (2001), 224--242.

\bibitem[CDW]{SnapPy}
Marc Culler, Nathan~M. Dunfield, and Jeffrey~R. Weeks, \emph{Snap{P}y, a
  computer program for studying the topology of $3$-manifolds}, Available at
  \url{http://snappy.computop.org/} (30/12/2016).

\bibitem[Con09]{conder:regmaps2}
Marston Conder, \emph{Regular maps and hypermaps of {E}uler characteristic
  {$-1$} to {$-200$}}, J. Combin. Theory Ser. B \textbf{99} (2009), no.~2,
  455--459.

\bibitem[CSRL01]{Cormen:Algorithms}
Thomas~H. Cormen, Clifford Stein, Ronald~L. Rivest, and Charles~E. Leiserson,
  \emph{Introduction to algorithms}, 2nd ed., McGraw-Hill Higher Education,
  2001.

\bibitem[DAS{\etalchar{+}}]{boost}
Beman Dawes, David Abrahams, Boris Sch{\"a}ling, et~al., \emph{The {B}oost
  {C}++ libraries}, \url{http://www.boost.org/}.

\bibitem[DHL15]{DHL}
Nathan~M. Dunfield, Neil~R. Hoffman, and Joan~E. Licata, \emph{Asymmetric
  hyperbolic {L}-spaces, {H}eegaard genus, and {D}ehn filling}, Math Research
  Letters \textbf{22} (2015), no.~6, 1679--1698, \eprint{arXiv:1407.7827}.

\bibitem[DT03]{dunfield:virthaken}
Nathan~M. Dunfield and William~P. Thurston, \emph{The virtual {H}aken
  conjecture: experiments and examples}, Geom. Topol. \textbf{7} (2003),
  399--441.

\bibitem[EP88]{EP}
D.~B.~A. Epstein and R.~C. Penner, \emph{Euclidean decompositions of noncompact
  hyperbolic manifolds}, J. Differential Geom. \textbf{27} (1988), no.~1,
  67--80.

\bibitem[Eve04]{Everitt:manifoldsFromPlatonicSolids}
Brent Everitt, \emph{3-manifolds from {P}latonic solids}, Topology Appl.
  \textbf{138} (2004), no.~1-3, 253--263.

\bibitem[FGG{\etalchar{+}}16]{FGGTV:tetrahedralCensus}
Evgeny Fominykh, Stavros Garoufalidis, Matthias Goerner, Vladimir Tarkaev, and
  Andrei Vesnin, \emph{A census of tetrahedral hyperbolic manifolds},
  Experiment. Math. \textbf{25} (2016), no.~4, 466--481,
  \eprint{arXiv:1502.00383}.

\bibitem[Goe15]{Goerner:RegTessLinkComps}
Matthias Goerner, \emph{Regular tessellation link complements}, Experiment.
  Math. \textbf{24} (2015), no.~2, 225--246, \eprint{arXiv:1406.2827}.

\bibitem[Goe16]{goerner:platonicCensusData}
\bysame, 2016, \url{http://unhyperbolic.org/platonicCensus/}.

\bibitem[HIK{\etalchar{+}}16]{hikmot}
N.~Hoffman, K.~Ichihara, M.~Kashiwagi, H.~Masai, S.~Oishi, and A.~Takayasu,
  \emph{Verified computations for hyperbolic 3-manifolds}, Experiment. Math.
  \textbf{25} (2016), no.~1, 66--78, \eprint{arXiv:1310.3410},
  \url{http://www.oishi.info.waseda.ac.jp/~takayasu/hikmot/}.

\bibitem[Hof14]{hoffman:dodecahedral}
Neil Hoffman, \emph{On knot complements that decompose into regular ideal
  dodecahedra}, Geometriae Dedicata \textbf{173} (2014), no.~1, 299--308.

\bibitem[LST08]{LuoSchleimerTillmann:virtualGeometricTrig}
Feng Luo, Saul Schleimer, and Stephan Tillmann, \emph{Geodesic ideal
  triangulations exist virtually}, Proc. Amer. Math. Soc. \textbf{136} (2008),
  no.~7, 2625--2630.

\bibitem[Mar15]{martelli:boundingGeometrically}
Bruno Martelli, \emph{Hyperbolic three-manifolds that embed geodesically},
  2015, \eprint{arXiv:1510.06325}, preprint.

\bibitem[MR03]{MR}
Colin Maclachlan and Alan~W. Reid, \emph{The arithmetic of hyperbolic
  3-manifolds}, Graduate Texts in Mathematics, vol. 219, Springer-Verlag, New
  York, 2003.

\bibitem[NR92a]{neumannReid:Topology90Arithmetic}
Walter~D. Neumann and Alan~W. Reid, \emph{Arithmetic of hyperbolic manifolds},
  Topology '90 ({C}olumbus, {OH}, 1990), Ohio State Univ. Math. Res. Inst.
  Publ., vol.~1, de Gruyter, Berlin, 1992, pp.~273--310.

\bibitem[NR92b]{neumannReid:SmallVolOrbs}
\bysame, \emph{Notes on {A}dams' small volume orbifolds}, Topology '90
  ({C}olumbus, {OH}, 1990), Ohio State Univ. Math. Res. Inst. Publ., vol.~1, de
  Gruyter, Berlin, 1992, pp.~311--314.

\bibitem[Rei91]{reid:arithmeticity}
Alan~W. Reid, \emph{Arithmeticity of knot complements}, J. London Math. Soc.
  (2) \textbf{43} (1991), no.~1, 171--184.

\bibitem[Sla17]{slavich:figEightBoundingGeometrically}
Leone Slavich, \emph{The complement of the figure-eight knot geometrically
  bounds}, Proc. Amer. Math. Soc. \textbf{145} (2017), 1275--1285,
  \eprint{arXiv:1511.08684}.

\bibitem[Thu98]{Th:see}
William~P. Thurston, \emph{How to see {$3$}-manifolds}, Classical Quantum
  Gravity \textbf{15} (1998), no.~9, 2545--2571, Topology of the Universe
  Conference (Cleveland, OH, 1997).

\bibitem[vdV09]{Roland:AKTG}
Roland van~der Veen, \emph{The volume conjecture for augmented knotted
  trivalent graphs}, Algebr. Geom. Topol. \textbf{9} (2009), no.~2, 691--722.

\bibitem[Wal11]{Walsh}
Genevieve~S. Walsh, \emph{Orbifolds and commensurability}, Interactions between
  hyperbolic geometry, quantum topology and number theory, Contemp. Math., vol.
  541, Amer. Math. Soc., Providence, RI, 2011, pp.~221--231.

\bibitem[WS33]{seifertWeber:ref}
C.~Weber and H.~Seifert, \emph{Die beiden {D}odekaederr{\"a}ume}, Mathematische
  Zeitschrift \textbf{37} (1933), no.~1, 237--253.

\end{thebibliography}

\end{document}